\documentclass[11pt,reqno]{amsart}
\usepackage{amssymb}
\usepackage{mathtools}
\usepackage{fullpage}
\newtheorem {theorem}{Theorem}
\newtheorem {lemma}{Lemma}
\newtheorem{proposition}[theorem]{Proposition}
\newtheorem{corollary}[theorem]{Corollary}
\newtheorem {question}{Question}

\newcommand\F{\mathbb{F}}

\begin{document}

	\title{A simple polynomial for a transposition over finite fields }
	
	\author{Amr Ali Abdulkader  Al-Maktry}
	\address{\hspace{-12pt}Department of Analysis and Number Theory (5010) \\
		Technische Universit\"at Graz \\
		Kopernikusgasse 24/II \\
		8010 Graz, Austria}
	\email{almaktry@math.tugraz.at}
	\urladdr{}
	
	
	\begin{abstract}
Let $q>2$, and  let $a$ and $b$ be two elements of the finite field $\F_q$ with $a\ne 0$.   Carlitz represented the 
 transposition $(0a)$ by a polynomial of degree $(q-2)^3$. In this note, we represent the transposition
	$(ab)$ by a polynomial of degree $q-2$. Also, we use this polynomial to construct polynomials that represent permutations of finite local rings with residue field $\F_q$.
			\end{abstract}
	
	\maketitle

In his proof of the main result of \cite{carltranspo}, Carlitz showed, for a non-zero element $a$ of the finite field $\F_q$
of $q>2$ elements, 
that  the transposition $(0a)$ can be induced by the following polynomial

\begin{equation}
	g_a(x) = -a^2\Bigl(\Bigl((x-a)^{q-2}+\frac{1}{a}\Bigr)^{q-2}-a\Bigr)^{q-2}.
\end{equation}
By direct substitution one easily see that the polynomial $g_a$ induces the transposition $(0a)$.  However,
Carlitz has never explained how he has constructed such a complicated polynomial. It seems that there is an ambiguous secret beyond this polynomial. This was my impression when I first met this polynomial while working on my master's thesis.  Nevertheless, the ambiguity of this polynomial attracted Zieve \cite{carlex} who revealed the secret of this polynomial in the end. He showed that $(01)$ can be induced by the polynomial  $f(x)=r ( r(r(x)))$,  where  $r(x)=1-x^{q-2}$, and then by using linear transformations to obtain the required polynomial representing $(0a)$. we remark here that the obtained polynomial via his procedure is equivalent to Carlitz polynomial $g_a$ and of degree $(q-2)^3$.   Later Ugoliny \cite{carlex2} noticed that $g_a$ can be deduced by using Hua’s identity.

In this note,  we obtain  a polynomial of degree $q-2$ representing the transposition $(ab)$ for any two different elements
$a$ and $b$ of the finite field $\F_q$. To be fair,
  our polynomial is a generalization of that of Martin \cite{simpletr}.  Martin  proved that the polynomial
\begin{equation}\label{Mart}
h(x)=x^{ p-2}+x^{ p-3}+\cdots+x^2+2x+1
\end{equation}
represents the transposition $(01)$ over the filed $\F_p$ for every odd prime $p$.
Further, he showed that polynomial
\begin{equation}
(b-a)\Big(\Big(\frac{x-a}{b-a}\Big)^{(p-2)}+ \cdots+\Big(\frac{x-a}{b-a}\Big)^2+2\Big(\frac{x-a}{b-a}\Big)+1\Big)+a
\end{equation}
induces the transposition $(ab)$ over $\F_p$. However, he overlooked that his argument is quite valid
for any finite field $\F_q$ with $q>2$. Indeed, let us write $\F_q =\{a_0,a_1,\ldots,a_{q-1}\}$ with $a_0=0$ and $a_1=1$. Then the polynomial $\prod_{i=0}^{q-1}(x-a_i)$ divides the polynomial $x^q-x$ since each $a_i$ is a root of  $x^q-x$. But then, since they are monic polynomials of the same degree, we must have $\prod_{i=0}^{q-1}(x-a_i)=(x^q-x)$. Thus,
\[x(x-1)\prod_{i=2}^{q-1}(x-a_i)=x(x^{q-1}-1)=x(x-1)(x^{q-2}+x^{q-1}+\cdots+x^2+x+1).\]
Hence, \[\prod_{i=2}^{q-1}(x-a_i)=x^{q-2}+\cdots+x^2+x+1,\] whence the polynomial
$l(x)=x^{q-2}+\cdots+x^2+x+1$   maps $a_i$ to $0$ for $i=2,\ldots,q-1$. It will not be hard now to see that the polynomial 
\begin{equation}\label{hak}
	f(x)=x^{q-2}+x^{q-1}+\cdots+x^2+2x+1
\end{equation}
induces the transposition $(01)$ (compare (\ref{hak}) with (\ref{Mart})).

Now let $a$ and $b$ be two different elements of $\F_q$ and consider the  polynomial $k(x)=l_2(f(l_1(x)))$ where $l_1(x)=\frac{x-a}{b-a}$ and $l_2(x)=(b-a)x+a$. Then, since $f$ represents the transposition $(01)$, we have for an element $c\in \F_q$ that
$k(c)=l_2(f(l_1(c)))= \begin{cases}l_2(f(0))=(b-a)1+a=b & \textnormal{if } c=a,\\
 	l_2(f(1))=(b-a)0+a=a & \textnormal{if } c=b,\\
 	l_2(f(\frac{c-a}{b-a}))=(b-a)\frac{c-a}{b-a}+a=c & \textnormal{if } c\ne a,b.	
 \end{cases}$

But this means that $k$ represents $(ab)$. Finally, direct calculations show that
\begin{equation}
k(x)=(b-a)\Big(\Big(\frac{x-a}{b-a}\Big)^{(q-2)}+ \cdots+\Big(\frac{x-a}{b-a}\Big)^2+2\Big(\frac{x-a}{b-a}\Big)+1\Big)+a.
\end{equation}
We have just proved the following Theorem.
\begin{theorem}\label{haky}
	Let $\F_q$ be a finite field with $q>2$ elements, and let $a$ and $b$ be two different elements of $\F_q$.
	Then the polynomial 
	\begin{equation}\label{hakigen}f_{a,b}(x)=(b-a)\Big(\Big(\frac{x-a}{b-a}\Big)^{(q-2)}+ \cdots+\Big(\frac{x-a}{b-a}\Big)^2+2\Big(\frac{x-a}{b-a}\Big)+1\Big)+a
	\end{equation}
	represents the transposition $(ab)$.  
\end{theorem} 
From now on let $R$ be a finite local ring with maximal ideal $M\ne \{0\}$ and residue filed $R/M=\F_q$.

Polynomials representing permutations are called permutation polynomials while the induced permutations are called polynomial permutations. 
Next, we intend to construct permutation polynomials over finite commutative local rings with residue field $\F_q$ employing the permutation polynomial of Theorem~\ref{haky}. For this purpose, we need the following celebrated criteria for permutation polynomials over finite local rings which is a special case of a more general result due to N\"obauer \cite{Nopercon}.  
\begin{lemma} \cite[Theorem~2.3]{Nopercon}\cite[Theorem~ 3]{Necha} \label{percert}
	Let $R$ be a finite local ring.   Let   $f\in R[x]$ and let $f'$ be its formal derivative. 
	Then $f$ is a permutation polynomial on $ R$ 
	if and only if:
	\begin{enumerate}
		\item $f$ induces a permutation of $R/M$;
		\item  for  each $r\in R$,   $f'(r)\ne 0\mod{M}$.
	\end{enumerate}
\end{lemma}
Also, we notice here that we can replace the elements of $\F_q$ with a complete system  of residue modulo $M$ from the elements of $R$. In this sense, we can represent a polynomial over  $\F_q$ by a polynomial over $R$. Clearly, this representation is not unique.

 Now
 we give a simple procedure  for constructing permutation polynomials on finite local rings by using permutation polynomials over finite fields.

\begin{proposition}\label{perconst}
	Let $R$ be a finite commutative local ring and $\F_q$ its residue field with $q=p^n$ for some prime number $p$. Let $f,g,l\in R[x]$ such that $f$ induces a permutation of $\F_q$, and   $g(r)\ne 0\mod M$ for every $r\in R$. Then the polynomial 
	\begin{equation}
		h(x)= f(x)+(f'(x)+g(x))(x^q-x)+pl(x)
	\end{equation}
is a permutation polynomial over $R$. That is, $h$ induces a permutation of $R$. 
	\end{proposition}
\begin{proof}
	Since $p\in M$ and $(x^q-x)$ maps $R$ into $M$, we have that $h$ and $f$ represent the same function over
	$R/M=\F_q$. But, then $h$ represents  a permutation of $\F_q$   since $f$ is a permutation polynomial on $\F_q$. This shows the first assertion of Lemma~\ref{percert} is satisfied.
Now,	differentiating $h$ yields, $h'(x) =  qx^{q-1}(f'(x)+g(x))-g(x)+pl'(x)$. Therefore, for every $r\in R$,
we have by our choice of $g$
\begin{equation*}
	h'(r) =  qr^{q-1}(f'(r)+g(r)) -g(r)+pl'(r)=-g(r)\ne 0 \mod M.
\end{equation*}	 
This verifies   the second assertion of Lemma~\ref{percert} and completes the proof.	\end{proof}
As we mentioned earlier given two different elements of $\F_q$, we can consider them as elements of $R$ using a complete system of residue modulo $M$. Hence, the polynomial $f_{a,b}$ of Theorem~\ref{haky} can be considered as a polynomial over $R$. So,   as a consequence of Theorem~\ref{haky} and Proposition~\ref{perconst}, we have the following corollary.
\begin{corollary}
	Let $a,b\in R$ with $a\ne b\mod M$. Let   $g,l \in R[x]$ such that $g(r)\ne 0\mod M$ for every $r\in R$.  Then the polynomial 
	\begin{equation}\label{myconst}
		h(x)= f_{a,b}(x)+(f_{a,b}'(x)+g(x))(x^q-x)+pl(x)
	\end{equation}
represents an odd permutation of $R$.
\end{corollary} 

The set of all polynomial permutations  of $R$ (permutations induced by polynomials over $R$), which we denote by $\mathcal{P}(R)$, is a subgroup of the symmetric group $S_R$ on the elements of $R$  (being a non-empty closed subset of a finite subgroup). It is well-known that this group is a proper subgroup of the symmetric group $S_R$ unless $R=\F_q $ when  in this case the group of polynomial permutations $\mathcal{P}(\F_q)$ is just the symmetric group $S_{\F_q}$.  It is evident that the set of all transpositions of $\F_q$ generates $\mathcal{P}(\F_q)$; that is the set of transpositions induced by  the  polynomials given in  Equation~(\ref{hakigen}) generates $\mathcal{P}(\F_q)$. Unfortunately,
 transpositions of  $ \F_q$  obtained by polynomials can not  be lifted into transpositions of $R$ through the construction of Proposition~\ref{perconst}. For instance, the polynomial $2x+1$ induces  the transposition $(01)$ over $\F_3$. However, it induces a permutation containing a cycle of length  greater than 2 over  $\mathbb{Z}/3^n\mathbb{Z}$ for every $n\ge 2$. 
 
 Finally, we close this note with a question concerning the relation between polynomial permutations induced by polynomials of the form~(\ref{myconst}) and the group of polynomial permutations $\mathcal{P}(R)$.
\begin{question}
	Let $A$ be the set of all polynomial permutations $R$ induced by polynomials constructed by Equation~ (\ref{myconst}).
Does the set $A$ generate the group  $\mathcal{P}(R)$?	
	\end{question}
\noindent {\bf Acknowledgment.}
The author is supported by the Austrian Science Fund (FWF):P 35788-N.

	\bibliographystyle{plain}
\bibliography{Carlpol}

\begin{thebibliography}{1}

\bibitem{carltranspo}
Leonard Carlitz.
\newblock Permutations in a finite field.
\newblock {\em Proc. Amer. Math. Soc.}, 4:538, 1953.

\bibitem{simpletr}
Greg Martin.
\newblock A simple polynomial for a simple transposition.
\newblock {\em Amer. Math. Monthly}, 115(1):57--60, 2008.

\bibitem{Necha}
Alexander~A. Nechaev.
\newblock Polynomial transformations of finite commutative local rings of
  principal ideals.
\newblock 27:425--432, 1980.
\newblock transl. from 27 (1980) 885-897, 989.

\bibitem{Nopercon}
Wilfried N\"{o}bauer.
\newblock Zur {T}heorie der {P}olynomtransformationen und
  {P}ermutationspolynome.
\newblock {\em Math. Ann.}, 157:332--342, 1964.

\bibitem{carlex2}
Simone Ugolini.
\newblock On the proof of a theorem by {C}arlitz.
\newblock {\em J. Group Theory}, 18(1):109--110, 2015.

\bibitem{carlex}
Michael~E. Zieve.
\newblock On a theorem of {C}arlitz.
\newblock {\em J. Group Theory}, 17(4):667--669, 2014.

\end{thebibliography}
\end{document}